\documentclass{amsart}

\usepackage{amssymb}
\usepackage{amsthm}
\usepackage{amsmath}
\usepackage[matrix, arrow]{xy}
\xyoption{arrow}
\theoremstyle{plain}\newtheorem{Theorem}{Theorem}%[section]
\theoremstyle{plain}
\theoremstyle{plain}
\theoremstyle{plain}
\theoremstyle{plain}\newtheorem{Proposition}[Theorem]{Proposition}
\theoremstyle{definition}
\theoremstyle{definition}
\theoremstyle{definition}
\theoremstyle{definition}
\theoremstyle{definition}
\theoremstyle{definition}\newtheorem{Remark}[Theorem]{Remark}

    \def\OG{{\mathcal{O}G}}  
    \def\OH{{\mathcal{O}H}}  
    \def\OP{{\mathcal{O}P}}

\def\CO{{\mathcal{O}}}    \def\OPE{{\mathcal{O}(P\rtimes E)}}

\def\Z{{\mathbb Z}}

\def\Aut{\mathrm{Aut}}

\def\f{\mathrm{f}}

    \def\tenC{\otimes_C}

\def\Irr{\mathrm{Irr}}      
\def\mf{\mathrm{mf}}
      \def\tenO{\otimes_{\mathcal{O}}}
\def\mod{\mathrm{mod}}

\def\Out{\mathrm{Out}}

\def\Pr{\mathrm{Pr}}

\def\Res{\mathrm{Res}}           

\def\sf{\mathrm{sf}}

\makeindex
\title[Strong Frobenius numbers for cyclic defect]{The strong Frobenius 
numbers for cyclic defect blocks are equal to one} 
\author[Linckelmann]{Markus Linckelmann}
\address{Department of Mathematics,
City, University of  London EC1V 0HB, United Kingdom}
\email{markus.linckelmann.1@city.ac.uk}

\thanks{This material is based upon work supported by the National 
Science Foundation under Grant No. DMS-1440140 while the author was in 
residence at the Mathematical Sciences Research Institute in Berkeley, 
California, during the Spring  2018 semester. In addition, the author 
would like to acknowledge support from EPSRC under the grant 
EP/M02525X/1.} 



\date{\today}

\begin{document}

\begin{abstract}
The purpose of this note is to provide a reference for the fact that 
the strong Frobenius number, in the sense of Eaton and Livesey, of a 
block of a finite group with a cyclic defect group is 
equal to $1$. This answers a question of Farrell and Kessar.
\end{abstract}

\maketitle

Calculating Frobenius numbers of block algebras originated in work
of Kessar \cite{KessarDonovan} as one of the steps towards
Donovan's finiteness conjecture, which predicts that for a fixed
finite $p$-group $P$, there are only finitely many Morita equivalence
classes of blocks with defect groups isomorphic to $P$. In \cite{Fa17},
\cite{FaKe}, the authors calculate Frobenius numbers of blocks of 
quasi-simple finite groups. The notion of Frobenius numbers has been 
refined in work of Eaton and Livesey \cite{EaLiv} with the purpose to 
take in to account the action on character groups of algebra 
isomorphisms between Galois conjugate blocks. 

Let $p$ be a prime and let $(K,\CO,k)$ be a $p$-modular system, with
$k$ algebraically closed of characteristic $p$ and $\CO$ absolutely 
unramified of characteristic $0$. Let $\bar K$ be an algebraic closure
of $K$. For $G$ a finite group and $B$ a block of $\OG$, we denote by 
$\Irr(G)$ the set of absolutely irreducible $\bar K$-valued characters 
of $G$, and by $\Irr(B)$ the subset of $\Irr(G)$ of absolutely 
irreducible characters associated with $B$. Since $\CO$ is unramified,
the automorphism of the field $k$ sending $\mu\in k$ to $\mu^p$ lifts 
uniquely to an automorphism of the ring $\CO$, denoted by $\sigma$. We 
extend this to a ring automorphism, still denoted by $\sigma$, of $\OG$ 
sending $\sum_{x\in G}\lambda_x x$ to 
$\sum_{x\in G} \sigma(\lambda_x) x$.

In what follows, modules are finitely generated unital left modules,
unless stated otherwise.
For $U$ an $\OG$-module, we denote by $U^\sigma$ the $\OG$-module which
is equal to $U$ as an abelian group and on which a group algebra 
element $\sum_{x\in G}\lambda_x x$ acts as the action of the element
$\sum_{x\in G}\sigma^{-1}(\lambda_x) x$; that is, $U^\sigma$ is the
image of $U$ under the self Morita equivalence $\Res_{\sigma^{-1}}$
of $\mod(\OG)$ as an abelian category (but not as an $\CO$-linear
category, since the action of scalars in $\CO$ get twisted by
$\sigma^{-1}$). If $U$ is a $B$-module, then clearly $U^\sigma$ is a
$\sigma(B)$-module. We use the analogous notation for bimodules.

Extend $\sigma$ to an automorphism $\bar\sigma$ of $\bar K$ such that 
$\bar\sigma$ acts as identity on all roots of unity of $p$-power order 
in $\bar K$ (see \cite[Lemma 2.2]{FaKe} for the existence of 
$\bar\sigma$). As before, we extend $\bar\sigma$ to a ring automorphism 
of $\bar KG$, still denoted by $\bar\sigma$. For $\chi$ a character of a 
$\bar KG$-module $X$ we denote by $\chi^{\bar \sigma}$ the character of 
the $\bar K G$-module $X^{\bar\sigma}$. It is easy to see that for 
$x\in$ $G$ we have $\chi^{\bar\sigma}(x)=$ $\bar\sigma(\chi(x))$, or 
more generally, $\chi^{\bar\sigma}=$ 
$\bar\sigma\circ\chi\circ\bar\sigma^{-1}$, where $\bar\sigma^{-1}$ is 
considered as a ring automorphism of $\bar KG$ and $\bar\sigma$ is 
considered as an automorphism of $\bar K$. Clearly if $\chi\in$ 
$\Irr(B)$, then $\chi^{\bar\sigma}\in$ $\Irr(\sigma(B))$. It is shown in 
\cite[Proposition 3.5]{EaLiv} that  $\chi^{\bar\sigma}(us)=$ 
$\chi(us^p)$, where $u$ is a $p$-element in $G$ and $s$ a $p'$-element 
in $C_G(u)$. 

Following \cite{KessarDonovan}, \cite[Definition 1.3]{BeKe}, the 
{\it Frobenius number of a block $B$ of $\OG$}, denoted by $\f(B)$, is 
the smallest positive integer $m$ such that $B\cong$ $\sigma^m(B)$ as 
$\CO$-algebras, and the {\it Morita-Frobenius number of $B$}, denoted 
by $\mf(B)$, is the smallest positive integer $m$ such that $B$ and 
$\sigma^m(B)$ are Morita equivalent. 
Following \cite[Definition 3.8]{EaLiv} the {\it strong Frobenius number 
over of the block algebra $B$ of $\OG$}, denoted by $\sf(B)$, is the 
smallest positive integer $m$ such that there exists an $\CO$-algebra 
isomorphism $\beta : B\cong$ $\sigma^m(B)$ whose scalar extension 
$\bar\beta$ to $\bar K\tenO B\cong$ $\bar K\tenO \sigma^m(B)$ induces 
the correspondence $\chi\mapsto$ $\chi^{\bar\sigma^m}$ between  
$\Irr(B)$ and $\Irr(\sigma^m(B))$, or more formally, which satisfies 
the equality of $\bar K$-valued central functions
$$\bar\sigma^m\circ\chi\circ\bar\sigma^{-m} = \chi\circ\bar\beta^{-1}$$
on $\bar K\tenO \sigma^m(B)$, with the above notational abuse of 
considering $\bar\sigma^{-m}$ as a ring automorphism of $\bar K G$, 
restricted to $\bar K\tenO \sigma^m(B)$, while 
$\bar\sigma^m$ is considered as an automorphism of $\bar K$. 

Frobenius numbers of blocks over $k$ may a priori be smaller, so it is 
important to keep track of the base ring, which for the purpose of the 
present note is always the ring $\CO$ defined above.
It is shown in \cite[Proposition 2.3]{Fa17} that Morita Frobenius 
numbers of blocks over $k$ with cyclic defect groups are $1$. We extend 
this to strong Frobenius numbers over $\CO$. This result is used in 
\cite{FaKe}, which is the motivation for the present note.

\begin{Theorem} \label{cyclic-sfO}
Let $G$ be a finite group and $B$ a block of $\OG$ having a cyclic
defect group $P$. We have $\sf(B)=1$.
\end{Theorem}

\begin{proof}
We use without further notice standard results on blocks with a cyclic 
defect group. For an expository account with references, see 
\cite[Chapters 11]{LiBookII}. 

If $P$ is trivial, then $B$ is a matrix algebra over $\CO$, and the
result holds trivially. Suppose that $P\neq$ $\{1\}$. 

Denote by $E$ the inertial quotient of $B$ with respect to a choice of 
a maximal $(G,B)$-Brauer pair. Then $E$ is a cyclic subgroup of 
$\Aut(P)$ of order dividing $p-1$. Thus the character values of 
$\eta\in\Irr(P\rtimes E)$ are invariant under the
automorphism $\bar\sigma$ of $\bar K$, because $\bar\sigma$ fixes all
$p$-power order roots of unity, and it also fixes roots of unity of
order dividing $p-1$. In other words, for any $\eta\in$ 
$\Irr(P\rtimes E)$ we have $\eta^{\bar\sigma}=\eta$, and
hence the identity algebra automorphism of $\OPE$ shows that 
$\sf(\OPE)=1$.

Let $Z$ be the subgroup of order $p$ of $P$, and set $H=$ $N_G(Z)$.
Denote by $C$ the block of $\OH$ with defect group $P$ corresponding 
to $B$. Then $C$ is Morita equivalent to $\OPE$. It follows from 
\cite[Proposition 3.12]{EaLiv} that $\sf(C)=1$. In other words,
there is an $\CO$-algebra isomorphism $\beta : C\cong$ $\sigma(C)$
whose coefficient extension to $\bar K$ induces on characters the 
correspondence $\eta\mapsto$ $\eta^{\bar\sigma}$ for $\eta\in$ 
$\Irr(C)$. Since irreducible Brauer characters in $C$ lift to ordinary 
irreducible characters, it follows that $\beta$ induces on isomorphism 
classes of simple modules the correspondence given by $T\mapsto$ 
$T^{\sigma}$ for any simple $k\tenO C$-module $T$. Since indecomposable 
$k\tenO C$-modules are uniserial, hence determined by their top 
composition factor and their length, it follows that the Morita 
equivalence induced by $\beta$ sends any $k\tenO C$-module $V$ to 
$V^{\sigma}$. 

Now truncated induction and restriction between $C$ and $B$ yields a 
stable equivalence of Morita type between $C$ and $B$ given by the 
$B$-$C$-bimodule $M=$ $b\OG c$ and its dual $M^*\cong$ $c\OG b$, where 
$b=$ $1_B$ and $c=$ $1_C$. 
Then $M^{\sigma}\cong$ $\sigma(b)\OG\sigma(c)$ and its dual yield a 
stable equivalence of Morita type between $\sigma(C)$ and $\sigma(B)$. 
In addition, by the above, we have an $\CO$-algebra isomorphism 
$\beta : C\cong$ $\sigma(C)$ inducing the correspondence $\eta\mapsto$ 
$\eta^{\bar\sigma}$ on $\Irr(C)$. Combining the three correspondences 
yields a stable equivalence of Morita type between $B$ and $\sigma(B)$ 
given by the $B$-$\sigma(B)$-bimodule
$$M \tenC {{_\beta}(M^{\sigma}})^*$$
and its dual. We follow a simple $k\tenO B$-module $S$ through this 
equivalence (ignoring projective summands). First, the truncated
restriction of $S$ to $C$ yields an indecomposable $k\tenO C$-module 
$V$. Under $\beta$, and thanks to the above observations, this 
corresponds to the $k\tenO \sigma(C)$-module $V^{\sigma}$. This 
module, in turn, corresponds to the simple $k\tenO \sigma(B)$-module 
$S^{\sigma}$ through truncated restriction from $\sigma(B)$ to 
$\sigma(C)$. In other words, the stable equivalence between $B$ and 
$\sigma(B)$ given by $M \tenC {{_\beta}(M^{\sigma}})^*$ sends every 
simple $k\tenO B$-module $S$ to the simple $k\tenO \sigma(B)$-module 
$S^{\sigma}$, plus possibly a projective summand.
It follows from \cite[Theorem 2.1]{Listable} that the bimodule
$M \tenC {{_\beta}(M^{\sigma}})^*$ has, up to isomorphism, a unique 
nonprojective $B$-$\sigma(B)$-bimodule summand $N$, and that the stable 
equivalence given by $N$ is in fact a Morita equivalence. Again by 
construction, this Morita equivalence sends any $k\tenO B$-module $U$ 
to $U^{\sigma}$.

Note in particular that this Morita equivalence preserves the 
dimensions of simple modules, and hence is induced by an
$\CO$-algebra isomorphism $\alpha : B\cong$ $\sigma(B)$ satisfying
$N\cong$ ${_\alpha{\sigma(B)}}$ as $B$-$\sigma(B)$-bimodules.
We need to show that under this isomorphism, a character $\chi\in$
$\Irr(B)$ corresponds to $\chi^{\bar\sigma}$. Denote by $\Pr(B)$ the
subgroup of $\Z\Irr(B)$ of projective generalised characters (that is,
of generalised characters in $B$ which vanish on all $p$-singular
elements in $G$) and by $L^0(B)$ the subgroup of all generalised
characters in $\Z\Irr(B)$ which are perpendicular to $\Pr(B)$ (or
equivalently, which vanish on all $p'$-elements in $G$). The group
$\Pr(B)\oplus L^0(B)$ has finite index in $\Z\Irr(B)$, and therefore
it suffices to show that any $\psi$ in this subgroup corresponds
to $\psi^{\bar\sigma}$ under the isomorphism $\alpha$. 

Since a projective indecomposable $B$-modules is uniquely determined, 
up to isomorphism, by its unique simple quotient, it follows that 
any projective character $\psi$ in $B$ corresponds under $\alpha$ to 
$\psi^{\sigma}$. We need to show the analogous statement for $\psi\in$
$L^0(B)$. Recall from above that under the scalar extension of 
$\beta$ to $\bar K\tenO C$, every generalised character $\eta$ of 
$\bar K\tenO C$ corresponds to the generalised character 
$\eta^{\bar\sigma}$. Moreover, the stable equivalence of
Morita type between $B$ and $C$ given by $M$ induces a partial
isometry $L^0(B)\cong$ $L^0(C)$ (this is well-known; see e. g. 
\cite{Dadecyclic} or also \cite[\S 3]{KeLionesimple} for an account 
using the above notation). 
A generalised character $\psi\in$ $L^0(B)$ vanishes
except possibly on elements whose $p$-part is nontrivial and
conjugate to an element in $P$. Note that $H$ contains the centralisers
in $G$ of all nontrivial elements in $P$. Thus $\psi$ is completely 
determined by its restriction to $H$. The restriction map from $G$ to
$H$ applied to a generalised character $\psi$ of $G$ commmutes trivially 
with the correspondence $\psi\mapsto$ $\psi^{\bar\sigma}$ and the 
analogous correspondence on generalised characters of $H$. Moreover, at 
the level of $H$, this correspondence sends the component of 
$\Res^G_H(\psi)$ belonging to $C$ to the component of 
$\Res^G_H(\psi^{\bar\sigma})=$ $(\Res^G_H(\psi))^{\bar\sigma}$
belonging to $\sigma(C)$. Thus the isometries $L^0(B)\cong$ $L^0(C)$
and $L^0(\sigma(B))\cong$ $L^0(\sigma(C))$ given by the truncated 
restrictions from $B$ to $C$ and from $\sigma(B)$ to $\sigma(C)$,
respectively, commute with the correspondence sending $\psi\in$ $L^0(B)$ 
to $\psi^{\bar\sigma}$ and the analogous correspondence for $L^0(C)$. 
This shows that $\alpha$ induces the correspondence $\psi\mapsto$ 
$\psi^{\bar\sigma}$ for all $\psi\in$ $L^0(B)$, hence for all 
generalised characters in $B$. The result follows.
\end{proof}

For the sake of completeness, we mention that a similar result holds for
blocks with a Klein four defect group (the proof in this case is an
immediate consequence of the structure of the source algebras of blocks
with a Klein four defect group).
 
\begin{Proposition} \label{Kleinfour-sfO}
Suppose that $p=2$. Let $G$ be a finite group and $B$ a block of $\OG$ 
having a Klein four defect group $P$. We have $\sf(B)=1$. 
\end{Proposition}

\begin{proof}
If $B$ is nilpotent, then the result holds by Remark \ref{nilp-sfO}
below. Suppose that $B$ is not nilpotent. Then $B$ is Morita equivalent 
to either $\CO A_4$ or the principal block algebra of $\CO A_5$
(cf. \cite[Corollary 1.4]{Likleinfour} or \cite[Chapter 12]{LiBookII} for
an expository account with more references on blocks with Klein four 
defect groups). 
Suppose that $B$ is Morita equivalent to $\CO A_4$. The automorphism 
$\bar\sigma$ of $\bar K$ raising odd order roots of unity to their 
square and fixing $2$-power order roots of unity exchanges the two  
nontrivial linear characters of $A_4$ and fixes all other characters. 
The $\CO$-algebra automorphism of $\CO A_4$ given by conjugation with a 
transposition in $S_4$ has the same effect on the characters of $A_4$.
Thus by \cite[Proposition 3.12]{EaLiv} we have $\sf(B)=$ 
$\sf(\CO A_4)=1$. Suppose that $B$ is Morita equivalent to the 
principal block of $\CO A_5$. The automorphism $\bar\sigma$ of $\bar K$ 
raising odd order roots of unity to their squares and fixing $2$-power 
order roots of unity exchanges the two degree $3$ characters of $A_5$ 
and fixes all other characters. The $\CO$-algebra automorphism given by 
conjugation with a transposition in $S_5$ has the same effect on the 
characters of $A_5$. As before, we get $\sf(B)=1$.
\end{proof}

\begin{Remark} \label{nilp-sfO}
For any nilpotent block $B$ of $\OG$ we have $\sf(B)=1$. Indeed, if $B$ 
is nilpotent, then $B$ is Morita equivalent to a defect group algebra 
$\OP$, and hence, by \cite[Proposition 3.12]{EaLiv}, we have $\sf(B)=$ 
$\sf(\OP)=1$, where the last equality follows from the (trivial) fact 
that $\zeta^{\bar\sigma}=$ $\zeta$ for $\zeta\in$ $\Irr(P)$.
\end{Remark}

\begin{Remark}
If $\Irr(B)$ has a rational valued character, then $B$ is invariant 
under $\sigma$, and hence $\f(B)=1$. Being invariant under $\sigma$ is 
not sufficient to conclude that $\sf(B)=1$, since $\sigma$ might 
act nontrivially on $\Irr(B)$, and one would need to find an 
$\CO$-algebra automorphism of $B$ inducing the same action on $\Irr(B)$. 
In particular, the strong Frobenius numbers of principal blocks are
not known, but it is shown in \cite[Proposition 3.11]{EaLiv} that they 
are bounded in terms of defect groups. If all characters in $\Irr(B)$ 
have rational values, then clearly $\sf(B)=1$. In particular, blocks 
of symmetric groups have strong Frobenius number $1$.
\end{Remark}

\begin{Remark} \label{stable-sfO}
It is not clear, whether strong Frobenius numbers are invariant under
arbitrary stable equivalences of Morita type. Here is the problem: if 
there is a stable equivalence of Morita type between blocks $B$ and $C$ 
of positive defect, then a simple $k\tenO B$-module $S$ corresponds to 
an indecomposable nonprojective $k\tenO C$-module $V$, and similarly, 
the simple $k\tenO \sigma^m(B)$-module $S^{\sigma^m}$ corresponds to the 
indecomposable nonprojective $k\tenO \sigma^m(C)$-module $V^{\sigma^m}$, 
where $m$ is a positive integer. The issue is that even if there is an 
$\CO$-algebra isomorphism $C\cong$ $\sigma^m(C)$ under which every 
simple $k\tenO C$-module $T$ corresponds to $T^{\sigma^m}$, it is not 
clear why under this isomorphism the module $V$ would have to correspond
to $V^{\sigma^m}$. In the proof of Theorem \ref{cyclic-sfO} above, we 
have made use of specific properties of stable equivalences between 
blocks with cyclic defect groups in order to conclude that $V$ does in 
indeed correspond to $V^{\sigma^m}$, for any $m\geq$ $1$. Slightly more 
generally, the argument in the proof of Theorem \ref{cyclic-sfO} shows 
that if there is a stable equivalence between a block $B$ of $G$ and a 
block $C$ of a subgroup $H$ of $G$ induced by truncated induction and
restriction, then in order to show that $\sf_\CO(B)\leq$ $m$, it 
suffices to show that there is an $\CO$-algebra isomorphism
$\beta : C\cong$ $\sigma^m(C)$ with the property that any
indecomposable $k\tenO C$-module $V$ corresponding to a simple
$k\tenO B$-module (under the stable equivalence between $B$ and $C$)
is mapped to the module $V^{\bar\sigma^m}$ under the Morita equivalence 
between $C$ and $\sigma^m(C)$ induced by $\beta$ 
\end{Remark}

Since the generalised decomposition matrix of a block $B$ of $\OG$ is 
nondegenerate, one can detect the action of $\bar\sigma$ on $\Irr(B)$ 
also in terms of the action of $\sigma$ on local pointed groups, by 
making use of Puig's description of generalised decomposition numbers 
in the form $\chi(u_\epsilon)$, where $u_\epsilon$ runs over a set of 
representatives of $G$-conjugacy classes of local pointed elements on 
$B$, where $\chi\in$ $\Irr(B)$, and where $\chi(u_\epsilon)=$ $\chi(uj)$
for some $j\in$ $\epsilon$; this number is independent of the coice of 
$j$ as $\chi$ is a class function. Note that this description implies 
the well-known fact that generalised decomposition numbers are 
$\Z$-linear combinations of $p$-power roots of unity, hence they are 
invariant under $\bar\sigma$. The ring automorphism of $\OG$ obtained 
from extending $\sigma$ fixes the group elements, hence stabilises 
fixed point subalgebras, relative traces, and hence permutes (local) 
points of subgroups of $G$. This can be used to bound strong Frobenius 
numbers in terms of certain source algebra isomorphisms.

\begin{Proposition} \label{sfOsourcealgebra}
Let $G$ be a finite group, $B$ a block of $\OG$ with defect group
$P$, let $i\in$ $B^P$ be a source idempotent of $B$ and set $A=$ 
$iBi$. Let $m$ be a positive integer. Then $\sigma^m(i)$ is a source
idempotent of the block $\sigma^m(B)$ in $(\sigma^m(B))^P$. Suppose that
there exists a source algebra isomorphism $\alpha : A\to$ $\sigma^m(A)$ 
as interior $P$-algebras such that $\alpha(\epsilon)=$ 
$\sigma^m(\epsilon)$ for any $u\in$ $P$ and any local point $\epsilon$ 
of $\langle u\rangle$ on $A$. Then $\alpha$ extends to an $\CO$-algebra 
automorphism $\beta : B\cong$ $\sigma^m(B)$, the class of $\beta$ in 
$\Out(B)$ is uniquely determined by the class of $\alpha$ in $\Out(A)$, 
and for any $\chi\in$ $\Irr(B)$ we have
$$\chi^{\bar\sigma^m} = \chi\circ\beta^{-1}\ .$$
In particular, we have $\sf(B)\leq$ $m$.
\end{Proposition}

\begin{proof}
By the remarks preceding the Proposition, $\sigma^m(i)$ is a source
idempotent of $\sigma^m(B)$ for the same defect group $P$. The
extendibility of $\alpha$, uniquely in its class of outer automorphisms,
follows from the fact that $A$ and $B$ are Morita equivalent, together
with the fact that a simple $k\tenO B$-module $S$ and its corresponding
simple $k\tenO \sigma(B)$-module $S^{\sigma^m}$ have the same dimension.
Let $\chi\in$ $\Irr(B)$ and let $u_\epsilon$ be a local pointed
element on $A$, with $u\in$ $P$. Then $\epsilon'=$ 
$\sigma^m(\epsilon)$ is a local point of $\langle u\rangle$ on 
$\sigma^m(A)$, and all local points of $\langle u\rangle$ on 
$\sigma^m(A)$ arise in this way.
Thus it suffices to show that $\chi^{\bar\sigma^m}(u_{\epsilon'})=$
$(\chi\circ\beta^{-1})(u_{\epsilon'})$. Now $\epsilon'=$ 
$\sigma^m(\epsilon)=$ $\beta(\epsilon)$. Choose $j\in$ $\epsilon$.
Then $\sigma^m(j)$ and $\beta(j)$ both belong to $\epsilon'$, by
the assumptions. Since $\sigma$ fixes group elements, we have 
$\sigma^m(uj)=$ $u\sigma^m(j)$. By the assumptions on $\beta$, we also
have $\beta(uj)=$ $u\beta(j)$. Thus we have
$$(\chi\circ\beta^{-1})(u_{\epsilon'}) = \chi(u_\epsilon)$$
and we also have
$$\chi^{\bar\sigma^m}(u_{\epsilon'}) = \bar\sigma^m(\chi(u_\epsilon))=
\chi(u_\epsilon)\ ,$$
where in the last equation we use that $\bar\sigma$ fixes 
$\chi(u_\epsilon)$. The result follows.
\end{proof}

\begin{Remark} 
One can use Proposition \ref{sfOsourcealgebra} to give an alternative
for the second half of the proof of Theorem \ref{cyclic-sfO}. With the 
notation and hypotheses of Theorem \ref{cyclic-sfO}, an indecomposable 
endopermutation $\OP$-module $V$ with vertex $P$ and determinant $1$ is 
invariant under $\sigma$. Since $B$ and $\sigma(B)$ are isomorphic as 
rings, it follows that the combinatorial data of $B$ and $\sigma(B)$ 
(expressed in terms of their Brauer trees) coincide. The classification 
of source algebras of blocks with cyclic defect groups implies 
that $B$ and $\sigma(B)$ have isomorphic source algebras, as interior
$P$-algebras. Any such isomorphism will satisfy the additional 
hypotheses in Proposition  \ref{sfOsourcealgebra} on the nontrivial 
local pointed elements for the simple reason that any nontrivial
subgroup of $P$ has a unique local point on a source algebra of $B$.
(This is a general feature of blocks with an abelian defect group and
Frobenius inertial quotient.)
The local points of the trivial subgroup of $P$ correspond to the
isomorphism classes of simple $k\tenO B$-modules, and the first part of
the proof of Theorem \ref{cyclic-sfO}, adapted to source 
algebras, shows that the additional hypotheses in Proposition 
\ref{sfOsourcealgebra} are also satisfied for the local points of the 
trivial subgroup.
\end{Remark}

%%%%%%%%%%%%%%%%%%%%%%%%%%%%%%%%%%%%%%%%%%%%%%%%%%%%%%%%%%%%%%%%%%

\end{document}